\documentclass[10pt]{amsart}
\usepackage{amsmath}
\usepackage{amssymb}
\usepackage{amsthm}
\usepackage{indentfirst}
\usepackage[latin1]{inputenc}
\usepackage{geometry}
\usepackage{amsfonts}
\usepackage{enumitem}
\usepackage{cancel}
\usepackage[bookmarks=false,colorlinks=true,linkcolor=black,citecolor=black,filecolor=black,urlcolor=black]{hyperref} 
\usepackage{verbatim,listings}
\DeclareSymbolFont{largesymbols}{OMX}{yhex}{m}{n}
\DeclareMathAccent{\widehat}{\mathord}{largesymbols}{"62}

\newcommand{\F}{{\mathbb{F}}}
\newcommand{\Q}{{\mathbb{Q}}}
\newcommand{\Z}{{\mathbb{Z}}}

\newcommand{\Aut}{\textnormal{Aut}}
\newcommand{\Gal}{\textnormal{Gal}}
\newcommand{\Cen}{\textnormal{Cen}}
\newcommand{\tr}{{\textnormal{tr}}}
\newcommand{\Ker}{{\textnormal{Ker\,}}}
\newcommand{\supp}{{\textnormal{supp}}}
\newtheorem{theorem}{Theorem}[section]
\newtheorem{lemma}[theorem]{Lemma}
\newtheorem{corollary}[theorem]{Corollary}
\newtheorem{remark}[theorem]{Remark}

\title[Finite group algebras of nilpotent groups]{Finite group algebras of nilpotent groups: a complete set of orthogonal primitive idempotents}

\author{Inneke Van Gelder}
\email{ivgelder@vub.ac.be}
\address{Department of Mathematics, Vrije Universiteit Brussel,
Pleinlaan 2, 1050 Brussels, Belgium}

\author{Gabriela Olteanu} \email{gabriela.olteanu@econ.ubbcluj.ro}
\address{Department of Statistics-Forecasts-Mathematics, Babe\c s-Bolyai University,
Str. T. Mihali 58-60, 400591 Cluj-Napoca, Romania}

\begin{document}
\maketitle

\begin{abstract}
We provide an explicit construction for a complete set of orthogonal primitive idempotents of finite group algebras over nilpotent groups.
Furthermore, we give a complete set of matrix units in each simple epimorphic image of a finite group algebra of a nilpotent group.
\end{abstract}

\section{Introduction}
The group algebra $F G$ of a finite group $G$ over a field $F$ is the ring theoretical tool that links finite group theory and ring theory.
If the order of $G$ is invertible in $F$, then $F G$ is a semisimple algebra and hence is a direct sum of matrices over division rings,
called the simple components in the Wedderburn decomposition of $F G$. A concrete realization of the Wedderburn decomposition is of interest in
many topics. The Wedderburn decomposition shows its importance in the investigations of the group of units of a group ring,
see for example \cite{Jespers2010}, and of the group of automorphisms of a group ring, see for example \cite{Olivieri2006}.
Further, computing the primitive idempotents of a group ring gives control of the representations in the base field inside the group ring.
Moreover, a complete set of orthogonal primitive idempotents gives us enough information to compute all one-sided ideals of the group ring up to
conjugation.
For more information on group rings, the interested reader is referred to the books of Passman and Sehgal \cite{Passman1977, Sehgal1978, Sehgal1993}.

Finite group algebras and their Wedderburn decomposition are not
only of interest in pure algebra, they also have applications in
coding theory. Cyclic codes can be realized as ideals of group
algebras over cyclic groups \cite{HB1998} and many other important
codes appear as ideals of noncyclic group algebras
\cite{HB1998,charpin,sabin}. In particular, the Wedderburn
decomposition is used to compute idempotent generators of minimal
abelian codes \cite{Ferraz2007}. Using a complete set of
orthogonal primitive idempotents, one would be able to construct
all left $G$-codes, i.e. left ideals of the finite group algebra
$F G$, which is a much richer class then the (two-sided)
$G$-codes.

In this paper, we are interested in the computation of a complete
set of orthogonal primitive idempotents in a semisimple finite group
algebra $\F G$, for $\F$ a finite field and $G$ a nilpotent group.
This problem is related to the Wedderburn decomposition of $\F G$.
The first step for the computation of the Wedderburn components is
to determine the primitive central idempotents of $\F G$. The
classical method to do this deals with characters of the finite
group $G$. In 2003, \cite{Jespers2003} gave a character-free
method to compute the primitive central idempotents of the
rational group algebra $\Q G$ for a nilpotent group. Later,
\cite{Olivieri2004} and \cite{Broche2007} extended and improved
this method for more classes of groups over both the rationals and
finite fields. Furthermore, the Wedderburn component associated to
a primitive central idempotent is described for a large class of
groups, including the nilpotent groups. This is a second step
toward a detailed understanding of the Wedderburn decomposition of
$\F G$. These results are summarized in Section
\ref{Preliminaries}. Further, \cite{Jespers2010} describes a
complete set of matrix units (in particular, a complete set of
orthogonal primitive idempotents) of each Wedderburn component of
the rational group algebra $\Q G$ of a nilpotent group $G$, a
third step in the description of $\Q G$. In Section
\ref{Primitive idempotents} we prove similar results for the
finite group algebra $\F G$. These results can be used to
implement the computation of a complete set of orthogonal
primitive idempotents for finite group algebras over nilpotent
groups.

\section{Preliminaries}\label{Preliminaries}
Let $F$ be an arbitrary field and $G$ an arbitrary finite group such that $F G$ is semisimple.
The notation $H\leq G$ (resp. $H\unlhd G$) means that $H$ is a
subgroup (resp. normal subgroup) of $G$. For $H\leq G$, $g\in G$
and $h\in H$, we define $H^g=g^{-1}Hg$ and $h^g=g^{-1}hg$.
Analogously, for $\alpha \in F G$ and $g\in G$,
$\alpha^g=g^{-1}\alpha g$. For $H\leq G$, $N_G(H)$ denotes the
normalizer of $H$ in $G$ and we set
$\widetilde{H}=|H|^{-1}\sum_{h\in H} h$, an idempotent of $F G$,
and if $H=\langle g\rangle$ then we simply write $\widetilde{g}$
for $\widetilde{\langle g\rangle}$.

The classical method for computing primitive central idempotents
in a semisimple group algebra $F G$ involves characters of the
group $G$. All the characters of any finite group are assumed to
be characters in $\overline{F}$, a fixed algebraic closure of the
field $F$. For an irreducible character $\chi$ of $G$,
$e(\chi)=\frac{\chi(1)}{|G|}\sum_{g\in G}\chi(g^{-1})g$ is the
primitive central idempotent of $\overline{F}G$ associated to
$\chi$ and $e_{F}(\chi)$ is the only primitive central idempotent
$e$ of $F G$ such that $\chi(e)\neq 0$. The field of character
values of $\chi$ over $F$ is defined as $F(\chi)=F(\chi(g) : g\in
G)$, that is the field extension of $F$ generated over $F$ by the
image of $\chi$. The automorphism group $\Aut(\overline{F})$ acts
on $\overline{F}G$ by acting on the coefficients, that is
$\sigma\sum_{g\in G} a_gg=\sum_{g\in G}\sigma(a_g)g$, for
$\sigma\in\Aut(\overline{F})$ and $a_g\in\overline{F}$. Following
\cite{Yamada1973}, we know that
$e_{F}(\chi)=\sum_{\sigma\in\Gal(F(\chi)/F)}\sigma e(\chi)$.

New methods for the computation of the primitive central
idempotents in a group algebra do not involve characters. The main
ingredient in this theory is the following element, introduced in
\cite{Olivieri2004}. If $K \unlhd H\leq G$, then let
$\varepsilon(H,K)$ be the element of $\Q H\subseteq \Q G$ defined
as
\begin{eqnarray*}
\varepsilon(H,K)&=&
\left\{\begin{array}{ll}
\widetilde{K} & \mbox{if } H=K, \\
\prod_{M/K\in \mathcal{M}(H/K)}(\widetilde{K}-\widetilde{M}) & \mbox{if } H\neq K,
\end{array}\right.
\end{eqnarray*}
where $\mathcal{M}(H/K)$ denotes the set of minimal normal
non-trivial subgroups of $H/K$. Furthermore, $e(G,H,K)$ denotes
the sum of the different $G$-conjugates of $\varepsilon(H,K)$. By
\cite[Theorem 4.4]{Olivieri2004}, the elements $\varepsilon(H,K)$
are the building blocks for the primitive central idempotents of
$\Q G$ for abelian-by-supersolvable groups $G$.

In this paper, we focus on finite fields. First, we introduce some
notations and results from \cite{Broche2007}. Let $\F=\F_{q^m}$
denote a finite field of characteristic $q$ with $q^m$ elements,
for $q$ a prime and $m$ a positive integer, and $G$ a finite group
of order $n$ such that $\F G$ is semisimple, that is $(q,n)=1$.
Throughout the paper, we fix an algebraic closure of $\F$, denoted
by $\overline{\F}$. For every positive integer $k$ coprime with
$q$, $\xi_k$ denotes a primitive $k$th root of unity in
$\overline{\F}$ and $o_k$ denotes the multiplicative order of
$q^m$ modulo $k$. Recall that $\F(\xi_k)\simeq\F_{q^{mo_k}}$, the
field of order $q^{mo_k}$.

Let $\mathcal{Q}$ denote the subgroup of $\Z_n^*$, the group of
units of the ring $\Z_n$, generated by the class of $q^m$ and
consider $\mathcal{Q}$ acting on $G$ by $s\cdot g=g^s$. The
$q^m$-cyclotomic classes of $G$ are the orbits of $G$ under the
action of $\mathcal{Q}$ on $G$. Let $G^*$ be the group of
irreducible characters in $\overline{\F}$ of $G$. Now let
$\mathcal{C}(G)$ denote the set of $q^m$-cyclotomic classes of
$G^*$, which consist of linear faithful characters of $G$.

Let $K\unlhd H\leq G$ be such that $H/K$ is cyclic of order $k$ and $C\in\mathcal{C}(H/K)$. If $\chi\in C$ and $\tr=\tr_{\F(\xi_k)/\F}$ denotes the field
trace of the Galois extension $\F(\xi_k)/\F$, then we set
$$\varepsilon_C(H,K)=|H|^{-1}\sum_{h\in H} \tr(\chi(hK))h^{-1}=[H:K]^{-1}\widetilde{K}\sum_{X\in H/K}\tr(\chi(X))h_X^{-1} ,$$
where $h_X$ denotes a representative of $X\in H/K$. Note that
$\varepsilon_C(H,K)$ does not depend on the choice of $\chi\in C$.
Furthermore, $e_C(G,H,K)$ denotes the sum of the different
$G$-conjugates of $\varepsilon_C(H,K)$. Note that the elements
$\varepsilon_C(H,K)$ will occur in Theorem \ref{mainfinite} as the
building blocks for the primitive central idempotents of finite
group algebras.

If $H$ is a subgroup of $G$, $\psi$ a linear character of $H$ and
$g\in G$, then $\psi^g$ denotes the character of $H^g$ given by
$\psi^g(h^g)=\psi(h)$. This defines an action of $G$ on the set of
linear characters of subgroups of $G$. Note that if $K=\Ker\psi$,
then $\Ker\psi^g=K^g$ and therefore the rule $\psi\mapsto\psi^g$
defines a bijection between the set of linear characters of $H$
with kernel $K$ and the set of linear characters of $H^g$ with
kernel $K^g$. This bijection maps $q^m$-cyclotomic classes to
$q^m$-cyclotomic classes and hence induces a bijection
$\mathcal{C}(H/K)\rightarrow\mathcal{C}(H^g/K^g)$. 

Let $K\unlhd H\leq G$ be such that $H/K$ is cyclic. Then the
action from the previous paragraph induces an action of
$N=N_G(H)\cap N_G(K)$ on $\mathcal{C}(H/K)$ and it is easy to see
that the stabilizer of a cyclotomic class in $\mathcal{C}(H/K)$ is
independent of the cyclotomic class. We denote by $E_G(H/K)$ the
stabilizer of such (and thus of any) cyclotomic class in
$\mathcal{C}(H/K)$ under this action.

\begin{remark}\rm
The set $E_G(H/K)$ can be determined without the need to use
characters. Let $K\unlhd H\leq G$ be such that $H/K$ is cyclic.
Then $N=N_G(H)\cap N_G(K)$ acts on $H/K$ by conjugation and this
induces an action of $N$ on the set of $q^m$-cyclotomic classes of
$H/K$. It is easy to verify that the stabilizers of all the
$q^m$-cyclotomic classes of $H/K$ containing generators of $H/K$
are equal and coincide with $E_G(H/K)$.
\end{remark}

There is a strong connection between the elements $\varepsilon(H,K)$ and $\varepsilon_C(H,K)$ given in the following Lemma from \cite{Broche2007}.

\begin{lemma}\label{sumC}
Let $\Z_{(q)}$ denote the localization of $\Z$ at $q$. We identify $\F_{q}$ with the residue field
of $\Z_{(q)}$, denote with $\overline{x}$ the projection of $x\in \Z_{(q)}$ in $\F_{q}\subseteq \F$ and extend this notation to the projection of
$\Z_{(q)}G$ onto $\F_{q}G\subseteq \F G$.
\renewcommand{\theenumi}{\arabic{enumi}}
\renewcommand{\labelenumi}{\arabic{enumi}.}
\begin{enumerate}
\item\label{sumC1} Let $K\unlhd H\leq G$ be such that $H/K$ is
cyclic. Then
$$\overline{\varepsilon(H,K)}=\sum_{C\in\mathcal{C}(H/K)}
\varepsilon_C(H,K).$$ \item\label{sumC2} Let $K\leq H\unlhd
N_G(K)$ be such that $H/K$ is cyclic and $R$ a set of
representatives of the action of $N_G(K)$ on $\mathcal{C}(H/K)$.
Then
$$\overline{e(G,H,K)}=\sum_{C\in R}e_C(G,H,K).$$
\end{enumerate}
\end{lemma}

A strong Shoda pair of $G$ is a pair $(H,K)$ of subgroups of $G$ satisfying the following conditions:
\begin{itemize}
\item[(SS1)] $K\leq H\unlhd N_G(K)$,
\item[(SS2)] $H/K$ is cyclic and a maximal abelian subgroup of $N_G(K)/H$, and
\item[(SS3)] for every $g\in G\setminus N_G(K)$, $\varepsilon(H,K)\varepsilon(H,K)^g=0$.
\end{itemize}

\begin{remark}\rm From \cite[Theorem 7]{Broche2007}, we know that there is a strong relation between
the primitive central idempotents in a rational group algebra $\Q
G$ and the primitive central idempotents in a finite group algebra
$\F G$ that makes use of the strong Shoda pairs of $G$. More precisely, if
$X$ is a set of strong Shoda pairs of $G$ and every primitive
central idempotent of $\Q G$ is of the form $e(G,H,K)$ for
$(H,K)\in X$, then every primitive central idempotent of $\F G$ is
of the form $e_C(G,H,K)$ for $(H,K)\in X$ and
$C\in\mathcal{C}(H/K)$.
\end{remark}

We will use the following description of the primitive central idempotents and the simple components for abelian-by-supersolvable groups,
given in \cite{Broche2007}.
\begin{theorem}\label{mainfinite}
If $G$ is an abelian-by-supersolvable group and $\F$ is a finite field of order $q^m$ such that $\F G$ is semisimple, then every primitive central
 idempotent of $\F G$ is of the form $e_C(G,H,K)$ for $(H,K)$ a strong Shoda pair of $G$ and $C\in\mathcal{C}(H/K)$. Furthermore, for every
strong Shoda pair $(H,K)$ of $G$ and every $C\in\mathcal{C}(H/K)$, $\F Ge_C(G,H,K)\simeq M_{[G:H]}(\F_{q^{mo/[E:K]}})$, where $E=E_G(H/K)$ and
 $o$ is the multiplicative order of $q^m$ modulo $[H:K]$.
\end{theorem}

The following Lemma was proved in \cite{Jespers2010}. The groups listed will be the building blocks in the proof of Theorem \ref{main}.
For $n$ and $p$ integers with $p$ prime, we use $v_p(n)$ to denote the valuation at $p$ of $n$, i.e. $p^{v_p(n)}$ is the maximal $p$-th power dividing $n$.
\begin{lemma}\label{p-group}
Let $G$ be a finite $p$-group which has a maximal abelian subgroup which is cyclic and normal in $G$. Then $G$ is isomorphic to one of the groups given by
the following presentations:
\begin{eqnarray*}
P_1 &=& \langle a,b \mid a^{p^n}=b^{p^k}=1,b^{-1}ab=a^r\rangle,\\ && \mbox{with either } v_p(r-1)=n-k \mbox{, or } p=2 \mbox{ and }r\not\equiv 1\mod 4,\\
P_2 &=& \langle a,b,c \mid a^{2^n}=b^{2^k}=c^2=1, bc=cb, b^{-1}ab=a^r, c^{-1}ac=a^{-1}\rangle, \\&&\mbox{with  } r\equiv 1\mod 4,\\
P_3 &=& \langle a,b,c \mid a^{2^n}=b^{2^k}=1, c^2=a^{2^{n-1}}, bc=cb, b^{-1}ab=a^r, c^{-1}ac=a^{-1}\rangle,\\ &&\mbox{with  } r\equiv 1\mod 4.
\end{eqnarray*}

Note that if $k=0$ (equivalently, if $b=1$), then the first case corresponds to the case when $G$ is abelian (and hence cyclic),
the second case coincides with the first case with $p=2$, $k=1$ and $r=-1$, and the third case is the quaternion group of order $2^{n+1}$.
\end{lemma}

\section{Primitive idempotents}\label{Primitive idempotents}
From Theorem \ref{mainfinite}, we know that the primitive central idempotents of a semisimple group algebra $\F G$, over a nilpotent group,
are of the form $e_C(G,H,K)$, for $(H,K)$ a strong Shoda pair of $G$ and $C\in \mathcal{C}(H/K)$. We will now describe a complete set of
orthogonal primitive idempotents and a complete set of matrix units of $\F Ge_C(G,H,K)$ for finite nilpotent groups $G$.

\begin{remark}\rm
Our aim is to construct a complete set of orthogonal primitive idempotents in each simple epimorphic image of a
finite group algebra of a nilpotent group.
Throughout this paper we consider semisimple group rings, because arbitrary finite group rings can be easily reduced to the semisimple case.
To see this, let $\F$ be a field of characteristic $q>0$, and let $G$ be a finite nilpotent group with Sylow $q$-subgroup $G_q$.
Then by \cite[Lemma 1.6]{Passman1977}, the radical $J(\F G_q)$ equals $\omega(\F G_q)$, the augmentation ideal of $\F G_q$.
It is now easy to see that $\F G/(J(\F G_q))\simeq \F(G/G_q)$, where $(J(\F G_q))$ denotes the ideal in $\F G$ generated by $J(\F G_q)$,
and the problem is reduced to the semisimple case, since $\F(G/G_q)$ is semisimple.
\end{remark}

\begin{remark}\label{sumOfSquares}\rm
Let $\F$ be a finite field of order $q^m$. Then every element of
$\F$ is a sum of two squares. To see this, fix $z\in \F$ and
consider the sets $A=\{x^2 : x\in \F\}$ and $B=\{z-y^2:y\in \F\}$.
Since they both contain more than $\frac{1}{2}q^m$ elements, there
exists an element $w\in A\cap B$ and hence, we find $x$ and
$y\in\F$ such that $x^2+y^2=z$. In particular, we find $x,y\in F$
such that $x^2+y^2=-1$.
\end{remark}

Now we state the main Theorem. As said in the introduction, a
similar result was previously proved for rational group algebras,
see \cite[Theorem 4.5]{Jespers2010}. We will use these results and
adapt them to the finite case. However, several technical problems
have to be overcome in non-zero characteristic.
\begin{theorem}\label{main}
Let $\F$ be a finite field of order $q^m$ and $G$ a finite
nilpotent group such that $\F G$ is semisimple. Let $(H,K)$ be a
strong Shoda pair of $G$, $C\in\mathcal{C}(H/K)$ and set
$e_C=e_C(G,H,K)$, $\varepsilon_C=\varepsilon_C(H,K)$,
$H/K=\langle\overline{a}\rangle$, $E=E_G(H/K)$. Let $E_2/K$ and
$H_2/K=\langle\overline{a_2}\rangle$ (respectively $E_{2'}/K$ and
$H_{2'}/K=\langle\overline{a_{2'}}\rangle$) denote the $2$-parts
(respectively $2'$-parts) of $E/K$ and $H/K$ respectively. Then
$\langle\overline{a_{2'}}\rangle$ has a cyclic complement
$\langle\overline{b_{2'}}\rangle$ in $E_{2'}/K$.

A complete set of orthogonal primitive idempotents of $\F Ge_C$ consists of the conjugates of $\beta_{e_C}=\widetilde{b_{2'}}\beta_2\varepsilon_C$ by the
elements of $T_{e_C}=T_{2'}T_2T_E$, where $T_{2'}=\{1,a_{2'},a_{2'}^2,\dots,a_{2'}^{[E_{2'}:H_{2'}]-1}\}$, $T_E$ denotes a right transversal of $E$ in $G$
and $\beta_2$ and $T_2$ are given according to the cases below.

\renewcommand{\theenumi}{(\arabic{enumi})}
\renewcommand{\labelenumi}{(\arabic{enumi})}
\renewcommand{\theenumii}{(\roman{enumii})}
\renewcommand{\labelenumii}{(\roman{enumii})}
\begin{enumerate}
\item If $H_2/K$ has a complement $M_2/K$ in $E_2/K$ then $\beta_2=\widetilde{M_2}$. Moreover, if $M_2/K$ is cyclic, then there exists $b_2\in E_2$
such that $E_2/K$ is given by the following presentation
$$\langle \overline{a_2},\overline{b_2}\mid \overline{a_2}\hspace{1pt}^{2^n}=\overline{b_2}\hspace{1pt}^{2^k}=1,
\overline{a_2}\hspace{1pt}^{\overline{b_2}}=\overline{a_2}\hspace{1pt}^r \rangle,$$
and if $M_2/K$ is not cyclic, then there exist $b_2,c_2\in E_2$ such that $E_2/K$ is given by the following presentation
$$\langle \overline{a_2},\overline{b_2},\overline{c_2}\mid \overline{a_2}\hspace{1pt}^{2^n}=\overline{b_2}\hspace{1pt}^{2^k}=\overline{c_2}\hspace{1pt}^2=1,
\overline{a_2}\hspace{1pt}^{\overline{b_2}}=\overline{a_2}\hspace{1pt}^r,
\overline{a_2}\hspace{1pt}^{\overline{c_2}}=\overline{a_2}\hspace{1pt}^{-1},
[\overline{b_2},\overline{c_2}]=1 \rangle,$$ with $r\equiv 1 \mod
4$ (or equivalently $\overline{a_2}\hspace{1pt}^{2^{n-2}}$ is
central in $E_2/K$). Then
\begin{enumerate}
\item\label{fid1i} $T_2=\{1,a_2,a_2^2,\dots, a_2^{2^k-1}\}$, if $\overline{a_2}\hspace{1pt}^{2^{n-2}}$ is central in $E_2/K$
(unless $n\leq 1$) and $M_2/K$ is cyclic; and
\item\label{fid1ii}  $T_2=\{1,a_2,a_2^2,\dots,a_2^{d/2-1},a_2^{2^{n-2}},a_2^{2^{n-2}+1},\dots,a_2^{2^{n-2}+d/2-1}\}$, where $d=[E_2:H_2]$, otherwise.
\end{enumerate}
\item\label{fid2}  If $H_2/K$ has no complement in $E_2/K$, then there exist $b_2,c_2\in E_2$ such that $E_2/K$ is given by the following presentation
\begin{eqnarray*}
\langle \overline{a_2},\overline{b_2},\overline{c_2}&\mid& \overline{a_2}\hspace{1pt}^{2^n}=\overline{b_2}\hspace{1pt}^{2^k}=1,
\overline{c_2}\hspace{1pt}^2=\overline{a_2}\hspace{1pt}^{2^{n-1}}, \overline{a_2}\hspace{1pt}^{\overline{b_2}}=\overline{a_2}\hspace{1pt}^r,\\
&& \overline{a_2}\hspace{1pt}^{\overline{c_2}}=\overline{a_2}\hspace{1pt}^{-1},[\overline{b_2},\overline{c_2}]=1 \rangle,
\end{eqnarray*}
with $r\equiv 1 \mod 4$. In this case, $\beta_2=\widetilde{b_2}\frac{1+xa_2^{2^{n-2}}+ya_2^{2^{n-2}}c_2}{2}$ and
$$T_2=\{1,a_2,a_2^2,\dots, a_2^{2^k-1},c_2,c_2a_2,c_2a_2^2,\dots,c_2a_2^{2^k-1}\},$$ with $x,y\in\F$, satisfying $x^2+y^2=-1$ and $y\neq 0$.
\end{enumerate}
\end{theorem}
\begin{proof}
Similar to the proof of \cite[Theorem 4.5]{Jespers2010}, we start
by making some useful reductions. Take $T_E$ a right transversal
of $E$ in $G$. We recall from the proof of \cite[Theorem
7]{Broche2007} that $\F Ge_C\simeq M_{[G:E]}(\F E\varepsilon_C)$
and that the $[G:E]$ conjugates of $\varepsilon_C$ by the elements
of $T_E$ are mutually orthogonal and they are thus the
``diagonal'' elements in the matrix algebra $\F Ge_C\simeq
M_{[G:E]}(\F E\varepsilon_C)$. Hence it is sufficient to compute a
complete set of orthogonal primitive idempotents for $\F
E\varepsilon_C$ and then add their $T_E$-conjugates in order to
obtain the primitive idempotents of $\F Ge_C$. So one may assume
that $G=E$ and hence $T_E=\{1\}$. Since $\Cen_G(\varepsilon_C)=E$,
by \cite[Lemma 4]{Broche2007}, also $e_C=\varepsilon_C$. Since
$G=E\leq N_G(K)$, we have that $G=E=N_G(K)$.

Then the natural isomorphism $\F G\widetilde{K}\simeq \F (G/K)$ maps $\varepsilon_C$ to $\varepsilon_C(H/K,1)$. So, from now on, we assume that $K=1$
and hence $H=\langle a\rangle$ is a cyclic maximal abelian subgroup of $G$, which is normal in $G$ and $e_C=\varepsilon_C=\varepsilon_C(H,1)$.

The map $a\varepsilon_C\mapsto \xi_{|H|}$ induces an isomorphism
$\phi:\F H\varepsilon_C \simeq \F(\xi_{|H|})$. If $G=H$, then $\F
G e_C\simeq \F(\xi_{|H|})$, a field. So $\varepsilon_C$ is the
only non-zero idempotent. This is case \ref{fid1i} and
$\beta_2=1=\widetilde{b_{2'}}$ and $T_2=\{1\}=T_{2'}$. So in the
remainder of the proof, we assume that $G\neq H$.

Using the description of $\F G e_C$ given in Theorem \ref{mainfinite}, one obtains a description of $\F Ge_C$ as a matrix ring
$M_{[G:H]}(\F_{q^{mo/[G:H]}})$, with $o$ the multiplicative order of $q^m$ modulo $|H|$.
Then a complete set of orthogonal primitive idempotents contains exactly $[G:H]$ elements, that is the size of the matrix algebra
$M_{[G:H]}(\F_{q^{mo/[G:H]}})$.

We first consider the case when $G$ is a $p$-group. Let $|H|=p^n$. Then $G$ and $H=\langle a\rangle$, satisfy the conditions of Lemma \ref{p-group} and
therefore $G$ is isomorphic to one of the three groups of this Lemma.

Before we consider the different cases, we make the following
general remarks. Recall that we denote with $\overline{x}$ the
projection of $x\in \Z_{(q)}$ in $\F_{q}\subseteq \F$ and extend
this notation to the projection of $\Z_{(q)}G$ onto
$\F_{q}G\subseteq \F G$. For a subgroup $M$ of $G$, note that
$\widetilde{M}$, interpreted as element in $\Q G$, belongs to
$\Z_{(q)}G$ and projects to  $\widetilde{M}$, interpreted as
element in $\F G$, which allows notations like
$\overline{\widetilde{M}}=\widetilde{M}$. Let $\varepsilon$ denote
$\varepsilon(H,K)$ and $e$ denote $e(G,H,K)$ as elements in $\Q
G$.  Note that $\varepsilon=e$, since
$\Cen_G(\varepsilon)=N_G(K)$, by \cite[Lemma 3.2]{Olivieri2004}.
Since $\F G\overline{e}=\F G \overline{\varepsilon}=
\bigoplus_{D\in \mathcal{C}(H)}\F G\varepsilon_D(H,1)$, by Lemma
\ref{sumC}, we can now consider the projection of $\Z_{(q)}Ge$
onto $\F G e_C$, for the chosen $C\in\mathcal{C}(H)$, by the
composition of the projection of $\Z_{(q)}Ge$ on $\F
G\overline{e}$ with the projection of $\F G\overline{e}$ on $\F
Ge_C$, obtained by multiplying with $e_C=\varepsilon_C$.
We will use this mapping to project primitive idempotents of $\Q Ge$ 
on $\F Ge_C$. This can be done because, as we will see, the primitive 
idempotents of $\Q Ge$ in the different cases are of the form 
$(\widetilde{M}\varepsilon)^t$ for a subgroup $M$ of $G$ and an element 
$t$ in $G$, and therefore, these idempotents have coefficients 
in $\Z_{(q)}$.

Assume first that $G\simeq P_1$ and $v_p(r-1)=n-k$. According to
case (1)(i) in \cite[Theorem 4.5]{Jespers2010}, we have a complete
set of orthogonal primitive idempotents of $\Q Ge$, given by the
conjugates of $\widetilde{b}\varepsilon$ by the elements
$1$,$a$,$\dots$,$a^{[G:H]-1}$, where $\langle b \rangle$ is a
complement of $H=\langle a\rangle$ in $G$. Now take the
projections into $\F Ge_C$. This leads us to the set
$$\{(\widetilde{b}\varepsilon_C)^t : t\in
T=\{1,a,\dots,a^{[G:H]-1}\} \}$$ of orthogonal idempotents in $\F
Ge_C$. Since $\sum_{t\in T} (\widetilde{b}\varepsilon)^t= e$, also
$\sum_{t\in T} (\widetilde{b}\varepsilon_C)^t= \sum_{t\in T}
(\widetilde{b}\overline{\varepsilon})^t e_C=\overline{e}e_C=e_C$.
Now we have $[G:H]$ orthogonal idempotents adding up to $e_C$, so
it suffices to prove that these idempotents are non-zero in order
to prove that these are primitive idempotents of $\F Ge_C$. Since
$\supp(\widetilde{b})=\langle b\rangle$,
$\supp(\varepsilon_C)\subseteq H$, and $G=H \rtimes\langle
b\rangle$, the element $\widetilde{b}\varepsilon_C$ cannot be
zero. Hence we have found a complete set of orthogonal primitive
idempotents, according to case \ref{fid1i} in the statement of the
Theorem.

Assume now that $G$ is still isomorphic to $P_1$, but with $p=2$ and $r\not\equiv 1\mod 4$. Let $[G:H]=2^k$.
Now consider the complete set of orthogonal primitive idempotents of $\Q Ge$, from case (1)(ii) in \cite[Theorem 4.5]{Jespers2010}, given by
the conjugates of $\widetilde{b}\varepsilon$ by the elements $1$,$a$,$\dots$,$a^{2^{k-1}-1}$,$a^{2^{n-2}}$,$a^{2^{n-2}+1}$,$\dots$, $a^{2^{n-2}+2^{k-1}-1}$,
where $\langle b \rangle$ is a complement of $H=\langle a\rangle$ in $G$. Take the projections into $\F Ge_C$. Then we obtain the set
$$\{(\widetilde{b}\varepsilon_C)^t : t\in T=\{1,a,\dots,a^{2^{k-1}-1},a^{2^{n-2}},a^{2^{n-2}+1},\dots,a^{2^{n-2}+2^{k-1}-1}\} \}$$
of orthogonal idempotents in $\F Ge_C$. With the same arguments as in the previous case, this set is a complete set of orthogonal primitive idempotents
of $\F Ge_C$, according to case \ref{fid1ii}.

Now assume that $G\simeq P_2$. Let $[G:H]=2^{k+1}$. By case
(1)(ii) in \cite[Theorem 4.5]{Jespers2010}, a complete set of
orthogonal primitive idempotents of $\Q Ge$, is given by the
conjugates of $\widetilde{\langle b,c\rangle}\varepsilon$ by the
elements
$1$,$a$,$\dots$,$a^{2^{k}-1}$,$a^{2^{n-2}}$,$a^{2^{n-2}+1}$,$\dots$,
$a^{2^{n-2}+2^{k}-1}$, where $\langle b,c \rangle$ is a complement
of $H=\langle a\rangle$ in $G$. Take the projections into $\F
Ge_C$, then we obtain the set $$\{(\widetilde{\langle
b,c\rangle}\varepsilon_C)^t : t\in
T=\{1,a,\dots,a^{2^{k}-1},a^{2^{n-2}},a^{2^{n-2}+1},\dots,a^{2^{n-2}+2^{k}-1}\}
\}$$ of orthogonal idempotents in $\F Ge_C$. Using the same
arguments as before, this set is a complete set of orthogonal
primitive idempotents of $\F Ge_C$, according to case
\ref{fid1ii}.

Now assume that $G\simeq P_3$. Let $[G:H]=2^{k+1}$. By case (2)(i)
in \cite[Theorem 4.5]{Jespers2010}, a complete set of orthogonal
primitive idempotents of $\Q Ge$, is given by the conjugates of
$\widetilde{b}\varepsilon$ by the elements
$1$,$a$,$\dots$,$a^{2^k-1}$, where $b$ and $c$ are as in the
presentation of $P_3$, and take the projections into $\F Ge_C$.
Then we obtain the set $\{(\widetilde{b}\varepsilon_C)^t: t\in
T=\{1,a,\dots,a^{2^k-1}\}\}$ of orthogonal idempotents in $\F
Ge_C$, which sum up to $e_C$. These idempotents are non-zero,
since $H$ and $b$ generate a semidirect product in $G$. Since $\F
Ge_C\simeq M_{2^{k+1}}(\F_{q^{mo/[G:H]}})$, we have to duplicate
the number of idempotents. By Remark \ref{sumOfSquares}, we can
find $x,y\in \F$ such that $x^2+y^2=-1$, with $y\neq 0$. Let
$f=\frac{1}{2}(1+xa^{2^{n-2}}+ya^{2^{n-2}}c)$ and
$1-f=\frac{1}{2}(1-xa^{2^{n-2}}-ya^{2^{n-2}}c)$, non-zero
orthogonal idempotents which sum up to $1$. Observe that
$1-f=f^c$. Now define $\beta=\widetilde{b}f\varepsilon_C$ and
$T=\{1,a,\dots,a^{2^k-1},c,ca,\dots,ca^{2^k-1}\}$. Since there
always exists an integer $i$ such that
$a^i\in\supp(\varepsilon_C)$, one can check that
$ba^{2^{n-2}}ca^i\in\supp(\beta)$, using the relations in $G$.
Therefore, the conjugates of $\beta$ with elements of $T$ are
non-zero orthogonal and
$$\{(\widetilde{b}f\varepsilon_C)^t : t\in
T=\{1,a,\dots,a^{2^k-1},c,ca,\dots,ca^{2^k-1}\} \}$$ is a complete
set of orthogonal primitive idempotents of $\F Ge_C$, according to
case \ref{fid2}.

Let us now consider the general case, where $G$ is not necessarily a $p$-group. Then we have to combine the odd and even parts.
Since $G$ is finite nilpotent, $G=G_2\times G_{p_1}\times\cdots\times G_{p_r}=G_2\times G_{2'}$, with $p_i$ an odd prime for every $i=1,\dots,r$.
Then $(H,1)$ is a strong Shoda pair of $G$ if and only if $(H_{p_i},1)$ is a strong Shoda pair of $G_{p_i}$, for every $i=0,\dots,r$, with $p_0=2$.
One can consider $\chi\in C\in\mathcal{C}(H)$ as $\chi=\chi_0\chi_1\cdots\chi_r$, with $\chi_i$ obtained by restricting $\chi$ to $H_{p_i}$.
Since $\chi$ is faithful, all $\chi_i$ are faithful and hence we can consider $C_i$ the cyclotomic class containing $\chi_i$ in $\mathcal{C}(H_{p_i})$.
Let $\varepsilon_{C_i}=\varepsilon_{C_i}(H_{p_i},1)$ and recall that the definition of $\varepsilon_{C_i}$ is independent of the choice of
the character in $C_i$.
Moreover
\begin{eqnarray*}
\prod_i \varepsilon_{C_i}&=& \prod_i \left( |H_{p_i}|^{-1}\sum_{h_i\in H_{p_i}}\tr_{\F(\xi_{|H_{p_i}|})/\F}(\chi_i(h_i))h_i^{-1}\right)\\
&=& |H|^{-1} \sum_{\stackrel{h=h_0h_1\cdots h_r\in H}{h_i\in H_{p_i}}} \left(\prod_i \tr_{\F(\xi_{|H_{p_i}|})/\F}(\chi_i(h_i)) \right)h^{-1}.
\end{eqnarray*}
Since $\F(\xi_{|H_{p_i}|})\simeq \F_{q^{mo_i}}$, where $o_i$ is
the multiplicative order of $q^m$ modulo $|H_{p_i}|$, it follows
by definition that
$$\prod_i \varepsilon_{C_i}= |H|^{-1} \sum_{\stackrel{h=h_0h_1\cdots h_r\in H}{h_i\in H_{p_i}}} \left(
\sum_{0\leq l_i<
o_i}\chi_0(h_0)^{q^{l_0}}\chi_1(h_1)^{q^{l_1}}\cdots
\chi_r(h_r)^{q^{l_r}} \right)h^{-1}.$$ 
Let $o$ be the multiplicative order of $q^m$ modulo $|H|$. Since
$|H_2|$,$|H_{p_1}|$,$\dots$,$|H_{p_r}|$ are coprime, $o=o_0o_1\cdots
o_r$, and by comparing the number of elements occurring in the
inner sum, we have
\begin{eqnarray*}
\prod_i \varepsilon_{C_i}&=& |H|^{-1}\sum_{\stackrel{h=h_0h_1\cdots h_r\in H}{h_i\in H_{p_i}}} \sum_{0\leq l< o}(\chi_0(h_0)\chi_1(h_1)\cdots \chi_r(h_r))^{q^l} h^{-1} \\
&=&|H|^{-1}\sum_{\stackrel{h=h_0h_1\cdots h_r\in H}{h_i\in H_{p_i}}} \tr_{\F(\xi_{|H|})/\F}(\chi_0(h_0)\chi_1(h_1)\cdots\chi_r(h_r)) h^{-1}\\
&=&\varepsilon_C.
\end{eqnarray*}
Now it follows that $\F Ge_C=\F(\prod_i G_{p_i}\varepsilon_{C_i})\simeq \bigotimes_i \F G_{p_i}\varepsilon_{C_i}$,
the tensor product over $\F$ of the simple algebras $\F G_{p_i}\varepsilon_{C_i}$. On the other hand, we have seen that
$\F G_{p_i}\varepsilon_{C_i}\simeq M_{[G_{p_i}:H_{p_i}]}(F_i)$, for finite fields $F_i$. Therefore,
$$\F Ge_C\simeq \bigotimes_i \F G_{p_i}\varepsilon_{C_i}\simeq \bigotimes_i M_{[G_{p_i}:H_{p_i}]}(F_i) \simeq M_{[G:H]}(F), $$ for a finite field $F$.
Hence a complete set of orthogonal primitive idempotents of $\F
Ge_C$ can be obtained by multiplying the different sets of
idempotents obtained for each tensor factor. Each $G_{p_i}$, with
$i\geq 1$, takes the form $\langle a_i\rangle \rtimes\langle
b_i\rangle\simeq P_1$ and so $G_{2'}=\langle
a_{2'}\rangle\rtimes\langle b_{2'}\rangle$, with $a_{2'}=a_1\cdots
a_r$ and $b_{2'}=b_1\cdots b_r$. Having in mind that
$a_i^{p_i^{k_i}}$ is central in $G_{p_i}$, one can easily deduce,
with the help of the Chinese Remainder Theorem, that the product
of the different primitive idempotents of the factors from the odd
part are the conjugates of
$\widetilde{b_{2'}}\varepsilon_{C_2'}(H_{2'},1)$ by the elements
of $T_{2'}=\{1,a_{2'},\dots,a_{2'}^{[G_{2'}:H_{2'}]-1}\}$ as
desired. The primitive idempotents of the even part give
us $\beta_2\varepsilon_{C_2}(H_{2},1)$ and $T_2$, in the different
cases. Hence, multiplying the primitive idempotents of the odd and
even parts will result in conjugating the element
$\widetilde{b_{2'}}\beta_2\varepsilon_C$ by the elements of
$T_{2'}T_2$.
\end{proof}

\begin{remark}\rm
Theorem \ref{main} provides a straightforward implementation in a programming language, for example in GAP \cite{GAP}.
Computations involving strong Shoda pairs and primitive central idempotents are already provided in the GAP package Wedderga \cite{wedderga}.
Nevertheless, in case \ref{fid2}, there might occur some difficulties finding
$x,y\in \F$ satisfying the equation $x^2+y^2=-1$. Note that here $\F$ has to be of odd order $q^m$.

If $q\equiv 1\mod 4$, then $y^2=-1$ has a solution in
$\F_q\subseteq \F$. Half of the elements $\alpha$ of $\F_q$
satisfy the equation $\alpha^{\frac{q-1}{2}}=-1$. So we can pick
an $\alpha\in \F_q$ at random and check if the equality is
satisfied. If not, repeat the process. When we have found such an
$\alpha$, then take $y=\alpha^{\frac{q-1}{4}}$ and $x=0$.

If $2$ is a divisor of $m$, then $y^2=-1$ has a solution in $\F_{q^2}\subseteq \F$ since $q^2\equiv 1\mod 4$.
Half of the elements $\beta$ of $\F_{q^2}$ satisfy the equation $\beta^{\frac{q^2-1}{2}}=-1$. Pick a $\beta\in \F_{q^2}$ randomly.
If the equality is satisfied, then take $y=\beta^{\frac{q^2-1}{4}}$ and $x=0$.

Now assume that $q\not\equiv 1\mod 4$ and $2\nmid m$. Recall the Legendre symbol $(a/q)$ for an integer $a$ and an odd prime $q$,
is defined as $1$ if the congruence $x^2\equiv a\mod q$ has a solution, as $0$ if $q$ divides $a$ and as $-1$ otherwise.
Using the Legendre symbol, we can decide whether an element is a square modulo $q$ or not and this can be effectively calculated using
the properties of the Jacobi symbol as explained in standard references as, for example, \cite{Ireland1990}.

Take now a random element $a\in \F_q\subseteq \F$ and check if both $a$ and $-1-a$ are squares in $\F_q$.
If so, then we can use the algorithm of Tonelli and Shanks or the algorithm of Cornacchia \cite{Cohen1993} to compute square roots modulo $q$ and
to find $x$ and $y$ in $\F_q\subseteq \F$ satisfying $x^2+y^2=-1$.
\end{remark}

Similar to \cite[Corollary 4.10]{Jespers2010}, one can obtain a
complete set of matrix units of a simple component of $\F G$.

\begin{corollary}
Let $G$ be a finite nilpotent group and $\F$ a finite field of
order $q^m$ such that $\F G$ is semisimple. For every primitive
central idempotent $e_C=e_C(G,H,K)$, with $(H,K)$ a strong Shoda
pair of $G$ and $C\in\mathcal{C}(H/K)$, set $E=E_G(H/K)$ and let
$T_{e_C}$ and $\beta_{e_C}$ be as in Theorem \ref{main}. For every
$t,t'\in T_{e_C}$, let $$E_{tt'}=t^{-1}\beta_{e_C}t'.$$ Then
$\{E_{tt'} : t,t'\in T_{e_C}\}$ gives a complete set of matrix
units in $\F G{e_C}$, i.e. $e_C=\sum_{t\in T_{e_C}}E_{tt}$ and
$E_{t_1t_2}E_{t_3t_4}=\delta_{t_2t_3}E_{t_1t_4}$, for every
$t_1,t_2,t_3,t_4\in T_{e_C}$. Moreover $E_{tt}\F GE_{tt}\simeq
\F_{q^{mo/[E:K]}}$, where $o$ is the multiplicative order of $q^m$
modulo $[H:K]$.
\end{corollary}

\section*{Acknowledgements}
The authors would like to thank the referees for pointing out a
gap in the previous submission and for various suggestions. They
also thank Eric Jespers and \'Angel del R\'io for helpful
discussions.

Research partially supported by the grant PN-II-RU-TE-2009-1
project ID\_303.

\bibliographystyle{amsalpha}
\bibliography{references}

\end{document}